\documentclass[12pt]{amsart}

\pdfoutput=1
\usepackage{latexsym}
\usepackage[centertags]{amsmath}
\usepackage{amsfonts}
\usepackage{amssymb}
\usepackage{amsthm}
\usepackage{newlfont}
\usepackage{graphics}
\usepackage{color}
\usepackage{multirow}
\usepackage{wrapfig}
\usepackage{rotating}

\usepackage[usenames,dvipsnames]{xcolor}
\usepackage{graphicx}

\usepackage{wrapfig}

\newtheorem{theo}{Theorem}
\newtheorem{defn}[theo]{Definition}
\newtheorem{exam}[theo]{Example}
\newtheorem{lem} [theo]{Lemma}
\newtheorem{cor}[theo]{Corollary}
\newtheorem{prop}[theo]{Proposition}
\newtheorem{rem}[theo]{Remark}

\newtheorem{algor}[theo]{Algorithm}

\makeatletter 
\@addtoreset{theo}{}\makeatother

\def\CT{\mathop{\mathrm{CT}}}

\def\Z{\mathbb{Z}}
\def\Q{\mathbb{Q}}

\def\D{\mathcal{D}}

\def\y{\mathbf{y}}

\def\K{\mathcal{K}}
\def\ind{\textrm{ind}}
\def\m{\mathfrak{m}}

\newcommand\ci[1]{\mathrm{\textcircled{#1}}}

\def\mrem{\;\mathrm{rem}}
\def\A{\mathop{\mathcal{A}}}

\title{A Combinatorial Decomposition of Knapsack Cones}

\author{
Guoce Xin$^{1,*}$ and Yingrui Zhang$^2$ and Zihao Zhang$^3$
}

 \address{ $^{1,2}$School of Mathematical Sciences, Capital Normal University,
 Beijing 100048, PR China}
 \email{$^1$\texttt{guoce\_xin@163.com}\ \& $^2$\texttt{zyrzuhe@126.com} \& $^2$\texttt{zyrzuhe@126.com}}
\date{November 9, 2020}

\begin{document}

\begin{abstract}
In this paper, we focus on knapsack cones, a specific type of simplicial cones that arise naturally in the context of the knapsack problem
$x_1 a_1 + \cdots + x_n a_n = a_0$. We present a novel combinatorial decomposition for these cones, named \texttt{DecDenu}, which aligns with Barvinok's unimodular cone decomposition within the broader framework of Algebraic Combinatorics. Computer experiments support us to conjecture that our \texttt{DecDenu} algorithm is polynomial
when the number of variables \( n \) is fixed. If true, \texttt{DecDenu}  will provide the first alternative polynomial algorithm for Barvinok's unimodular cone decomposition,
at least for denumerant cones.

The \texttt{CTEuclid} algorithm is designed for MacMahon's partition analysis, and is notable for being the first algorithm to solve the counting problem for Magic squares of order 6. We have enhanced the \texttt{CTEuclid} algorithm by incorporating \texttt{DecDenu}, resulting in the \texttt{LLLCTEuclid} algorithm. This enhanced algorithm makes significant use of LLL's algorithm and stands out as an effective elimination-based approach.
\end{abstract}

\maketitle

\noindent
\begin{small}
 \emph{Mathematic subject classification}: Primary 52C07, Secondary 05--04, 05--08
\end{small}

\noindent
\begin{small}
\end{small}

\section{Introduction}
This paper focus on the knapsack problem \( x_1 a_1 + \cdots + x_n a_n = a_0 \), where \( a_0, a_1, \dots, a_n \) are positive integers
with \( \gcd(a_1, \dots, a_n) = 1 \). The number of nonnegative integer solutions to this problem, \( d(a_0; a_1, \dots, a_n) \), is known as Sylvester's denumerant.
The truth of \( d(a_0; a_1, \dots, a_n) = 0 \) is an NP-complete problem, and it is intricately related to the well-known Frobenius
numbers defined by $F(a_1,\dots,a_n)=\max \{a_0: d(a_0;a_1,\dots, a_n)=0\}$  \cite{alfonsin2005diophantine}.

Our primary objective is to introduce a combinatorial decomposition, akin to Barvinok's unimodular cone decomposition, for a class of simplicial cones
referred to as \emph{knapsack cones} or \emph{denumerant cones}. The denumerant (simplicial) cone is defined by
\begin{equation}
 \D_{\ell}(a_1, a_2, \dots, a_n) = \left\{ \sum_{i=1, i\neq \ell}^n k_i (-a_i e_{\ell} + a_{\ell} e_i) : k_i \in \mathbb{R}_{\geq 0} \right\},
\end{equation}
where \( e_i = (0, \dots, 0, 1, 0, \dots, 0)^T \) is the standard \( i \)-th unit vector for all \( i \). This cone is generated by \( d = n - 1 \)
vectors in \( \mathbb{R}^n \) and is therefore not full-dimensional.

When \( n \) is fixed,  \( d(a_0; a_1, \dots, a_n) \)  can be computed using Barvinok's polynomial algorithm\cite{barvinok1994polynomialalgorithm} from the field of Computational Geometry
for the lattice point counting of rational convex polytopes. The C-package \texttt{LattE} \cite{de2004effectiveLattE} is a notable implementation of Barvinok's algorithm,
often lauded as ``state of the art" and setting a high standard. Barvinok's algorithm, along with its \texttt{LattE} implementation, has been widely utilized across various studies, as indicated by the works of \cite{bajo2024weighted,  Baldoni2012, barvinokwodd2003short, breuer2017polyhedral, de2005computational,de2009ehrhartTodd,  DONGRE202339, xinxudedekindsums2023}, among others.  Unfortunately, it has ceased development.

Our approach falls under the domain of Algebraic Combinatorics and is in line with the \texttt{CTEuclid} algorithm developed in \cite{xin2015euclid},
which enables us to handle Laurent polynomial numerators uniformly. In that paper, the first author set forth the ambitious goal
of devising an optimal algorithm for MacMahon's partition analysis\cite{andrews2001macmahonOmegaalgorithm2}. Xin proposed to achieve this goal by integrating the elegant ideas of
LattE into his (non-polynomial) \texttt{CTEuclid} algorithm.
In the in-depth study of \texttt{LattE} and \texttt{CTEuclid},  we are confident that this goal will be achieved in the near future, with this paper contributing as a foundation.


We introduce a basic building block that relies on the following established result.
\begin{prop}\label{prop-PFD-fund}
Suppose the partial fraction decomposition of an  Elliott rational function  $F(\lambda)$ is given by
\begin{align}\label{f-partialfraction}
  F(\lambda)=\frac {L(\lambda)} {\prod_{i=1}^n(1-u_i\lambda^{a_i})}=P(\lambda) + \frac {p(\lambda)}{\lambda^k}+\sum_{i=1}^n\frac{A_i(\lambda)}{1-u_i\lambda^{a_i}},
\end{align}
where the $u_i$'s are free of $\lambda$, $L(\lambda)$ is a Laurent polynomial, $P(\lambda), p(\lambda)$, and the $A_i(\lambda)$'s are all polynomials,
$\deg p(\lambda) < k$, and  $\deg A_i(\lambda) < a_i$ for all $i$.

Then $A_s(\lambda)$ is uniquely characterized by the following properties: i) $\deg A_s(\lambda) < a_s$; ii) $
            A_s(\lambda)\equiv F(\lambda)\cdot (1-u_s \lambda^{a_s})  \pmod{\langle 1-u_s\lambda^{a_s}\rangle}$
  where $\langle f(\lambda) \rangle$ denotes the ideal generated by the polynomial $f(\lambda)$.
\end{prop}
This proposition applies when the denominator factors \( (1-u_i\lambda^{a_i}) \) are coprime to each other.

The fundamental construction we employ is a linear operator \( \A_{1-u_i\lambda^{a_i}}: F(\lambda) \mapsto A_i(0) \),
where the \( \lambda \) will be evident from the context. The \texttt{CTEuclid} algorithm utilizes recursion to compute \( A_i(0) \).

For simplicity, let us elucidate this concept using the denumerant model. We first express
$$d(a_0;a_1,\dots, a_n) = \CT_\lambda F(\lambda;\y) \big|_{y_i=1} \quad \text{where} \quad F(\lambda;\y) = \frac{\lambda^{-a_0}}{(1-\lambda^{a_1}y_1)\cdots (1-\lambda^{a_n} y_n)}.$$
Proposition \ref{prop-PFD-fund} then comes into play, yielding
$$ \CT_{\lambda} F(\lambda;\y) = \sum_{i=1}^n A_i(0) = \sum_{i=1}^n \A_{1-y_i\lambda^{a_i}} F(\lambda;\y).$$
For each \( i \), there exists a proper rational function \( \bar{F}(\lambda,\y) \) that vanishes at \( \lambda = 0 \) and is of the form
\begin{equation}\label{e-func-mod}
 \bar{F}(\lambda,\y) = \frac{u_0 \lambda^{b_0}}{(1-u_1\lambda^{b_1})\cdots (1-u_n\lambda^{b_n})},
\end{equation}
such that i) \( \bar{F}(\lambda;\y) \cdot (1-y_i\lambda^{a_i}) \equiv F(\lambda;\y) \cdot (1-y_i\lambda^{a_i}) \pmod{\langle 1-y_i\lambda^{a_i}\rangle} \);
ii) \( 0 \leq b_j < a_i/2 \) for all \( j \neq i \) and \( b_j \equiv \pm a_j \pmod{a_i} \).

Then we have the recursion
\begin{equation}\label{e_fund_recursion}
 \A_{1-y_i\lambda^{a_i}} F(\lambda;\y)=\A_{1-y_i\lambda^{a_i}} \bar{F}(\lambda;\y)=-\sum_{j\neq i} \A_{1-u_j\lambda^{b_j}} \bar{F}(\lambda;\y).
\end{equation}

In this way, we can give a decomposition of the form
\begin{equation}\label{sec1_sum_of_function}
  \A_{1-y_s\lambda^{a_s}} F(\lambda;\y)=\sum_{i=1}^{N_s} \frac{L_i(\y)}{(1-\y^{\alpha_{i,1}})\cdots (1-\y^{\alpha_{i,d}})},
\end{equation}
where $N_s$ is a positive integer and $d=n-1$ refers to the dimension of the problem.
By representing \( \CT_{\lambda} F(\lambda;\y) \) as a sum of \( N_1 + \cdots + N_n \) simple rational functions, we can efficiently compute
its limit at \( \y = 1 \) \cite{XinTodd2023}. However, it is important to note that \( N_s \) cannot be bounded by a polynomial in
the size \( \log(a_s) \) of the input, even if \( n \) is fixed.

Through a meticulous examination of \texttt{LattE}, we uncover a crucial connection (see \cite{xinxudedekindsums2023}). This connection reveals that:
\begin{equation*}
  \A_{1-y_s\lambda^{a_s}} F(\lambda;\y) = \sum_{\alpha \in (\D^v_{s}(a_1,a_2,\dots, a_n) \cap \Z^n)} \y^{\alpha},
\end{equation*}
where $\D^v_{s}(a_1,a_2,\dots, a_n)$ denotes the set $\D_{s}(a_1,a_2,\dots, a_n)$ shifted by the vector $v \in \Q^n$.

Barvinok's algorithm provides a unimodular decomposition of \( \D_{a_s}(a_1,a_2,\dots, a_n) \), thereby yielding a similar but ``short"  formula
akin to \eqref{sec1_sum_of_function}. In this context, ``short" refers to the fact that \( N_s \) is bounded by a polynomial in \( \log a_s \).

Our novel decomposition method is rooted in \texttt{CTEuclid}, augmented by an additional rule. This rule, encapsulated in
Proposition \ref{prop-key-transfer}, is pivotal and is equivalent to the assertion that:
\begin{equation}\label{e-A-multiplier}
  \A_{1-u\lambda^a} F(\lambda)(1-u\lambda^a) = \A_{1-u^{1/m}\lambda^a} F(\lambda^m)(1-u \lambda^{am}), \quad \text{if } (a,m) = 1.
\end{equation}
The multiplier \( m \) is a parameter that influences the ``shortness" of the resulting sum. For an optimal multiplier
\( \mathfrak{m} \) (where \( (a,\mathfrak{m}) = 1 \)), the application of \eqref{e-A-multiplier} followed by \eqref{e_fund_recursion}
results in small indices \( b_j \) that are bounded by \( a_s^{\frac{d-1}{d}} \),
a bound that aligns closely with Barvinok's algorithm. This suggests that we can expect a short sum.
Nevertheless, the condition \( (a,\mathfrak{m}) = 1 \) cannot be guaranteed, which can diminish
the "shortness" but is counterbalanced by several favorable properties that promote it.

We have developed the DecDenu C and Maple package to implement the aforementioned ideas.
The well-established LLL(Lenstra-Lenstra-Lovász) algorithm is employed to identify an effective multiplier,
a method akin to that used in \texttt{LattE}. The performance of our approach rivals that of \texttt{LattE} in the decomposition of denumerant cones,
leading us to conjecture that our algorithm is polynomial in the input size.
If this conjecture holds true, our decomposition method will provide the first alternative polynomial algorithm for Barvinok's unimodular cone decomposition,
at least for denumerant cones.

The structure of this paper is as follows. In Section 2, we delineate \texttt{LattE}'s implementation of Barvinok's unimodular cone decomposition algorithm.
This algorithm has set a high standard for the decomposition of denumerant cones. In Section 3, we present the key transformation and elucidate the process for determining a suitable multiplier.
Additionally, we provide Algorithm \texttt{DecDenu} and compare its performance with \texttt{LattE}.
Section 4 offers a concise introduction to the domain of iterated Laurent series and details the algorithm \texttt{LLLCTEuclid},
which addresses the core problem of MacMahon's partition analysis. The central issue here is the decomposition of denumerant cones.

\section{LattE's Implementation of Barvinok's Decomposition} \label{sec_BarAlg}
This section provides a concise overview of the fundamental principles underlying Barvinok's unimodular cone decomposition and its implementation in \texttt{LattE},
which inspires our development of DecDCone. For a detailed treatment, we direct the reader to \cite{de2004effectiveLattE}, and for
a deeper exploration within the realm of algebraic combinatorics, we refer to \cite{xin2023algebraic}.

Consider a nonsingular rational matrix \( A \) of order \( d \). This matrix
defines a lattice \( L(A) = A\mathbb{Z}^d \) and a simplicial cone \( C(A) = A \mathbb{R}_{\geq 0}^d \).
The LLL algorithm plays a crucial role in identifying a reduced basis for \( L(A) \), which is essential for locating the smallest
vector within the lattice under the infinity norm.

A simplicial cone \( \mathcal{K} = C(A) \) can be normalized to \( C(\bar{A}) \), where \( \bar{A} \) consists of primitive vectors
(i.e., integral vectors with gcd $1$).
The index \( \ind(\mathcal{K}) \) is then determined by \( |\det (\bar{A})| \), which counts the lattice points within
the fundamental parallelepiped \( \Pi(\mathcal{K}) = \{ k_1 \bar{\alpha}_{1} + \cdots + k_{d} \bar{\alpha}_{d} :
0 \leq k_1, \dots, k_d < 1 \} \). The lattice point generating function
\( \sigma_{\mathcal{K}}(\mathbf{y}) = \sum_{\alpha \in \mathcal{K} \cap \mathbb{Z}^d} \mathbf{y}^{\alpha} \)
is closely related to that of \( \Pi(\mathcal{K}) \) through Stanley's formula \cite{stanley2011EC1}:
\[
\sigma_{\mathcal{K}}(\mathbf{y}) = \frac{\sigma_{\Pi(\mathcal{K})}(\mathbf{y})}{\prod_{k=1}^{d} (1-\mathbf{y}^{\bar{\alpha}_{k}})}.
\]
Notably, \( \sigma_{\Pi(\mathcal{K})}(\mathbf{y}) = 1 \) if \( \mathcal{K} \) is unimodular, i.e., \( \ind(\mathcal{K}) = 1 \).

Barvinok's theorem provides a polynomial-time algorithm for dissecting a rational cone into unimodular cones, which is instrumental
in computationally addressing simplicial cones with large indices.
\begin{theo}[\cite{barvinok1994polynomialalgorithm}] \label{theo-polytime-barv}
Let \( d \) be the fixed dimension. There exists a polynomial-time algorithm that, for a given rational simplicial
cone \( \mathcal{K} = C(A) \subset \mathbb{R}^d \),
computes unimodular cones \( \mathcal{K}_i = C(A_i) \), \( i \in I = \{1,2,\dots, N\} \) and signs \( \epsilon_i \in \{ 1, -1 \} \) such that
\[
\sigma_{\mathcal{K}}(\mathbf{y}) = \sum_{i \in I} \epsilon_i \sigma_{\mathcal{K}_i}(\mathbf{y}).
\]
The sum is ``short" or polynomial in size, implying that the number of summands \( |I| = N \) is bounded by a polynomial in \( \log(\ind(\mathcal{K}) ) \).
\end{theo}
LattE's implementation of Barvinok's decomposition revolves around finding a lattice point \( \gamma \) within a closed
parallelepiped defined by \( \K \). This is achieved by computing the smallest vector \( \beta \) within the
lattice \( L(\bar{A}^{-1}) \) and setting \( \gamma = \bar{A}\beta \).
If \( \gamma \) meets the necessary condition, it is then utilized to decompose \( \mathcal{K} \) into
at most \( d \) cones \( \mathcal{K}_i \) with reduced indices that satisfy \( \ind (\mathcal{K}_i) \leq (\ind (\mathcal{K}))^{\frac{d-1}{d}} \).

The algorithm proceeds iteratively until all cones are reduced to unimodular cones.
 Utilizing Brion's polarization trick facilitates this decomposition by disregarding lower-dimensional cones.

\section{A Combinatorial Decomposition of Denumerant Cones}\label{sec_DecDenu}
We first describe how to find a good multiplier, which is defined after the key transformation in Proposition \ref{prop-key-transfer}. Then we
give the pseudo code of the algorithm.

\subsection{The Key Transformation and Good Multipliers}
For the function $F(\lambda)$ as in \eqref{f-partialfraction}, we also denote by
$$\CT_{\lambda} \underline{\frac{1}{1-u_s {\lambda}^{a_s}}}F(\lambda) (1-u_s {\lambda}^{a_s})=A_s(0)=\A_{1-u_s {\lambda}^{a_s}} F(\lambda).$$
For this notation, one can think that when taking the constant term in $\lambda$, only the single underlined factor
of the denominator contributes.

Let $(k,a)=1$, and $\xi= \exp(\frac{2\pi \sqrt{-1}}{a})$ be the $a$-th root of unity.
There is a basic fact that the map $x\mapsto x^k$ is an automorphism on the group $\{ \xi^j: j=0,1,\dots, a-1\} $
of $a$th roots of unity.
We establish the following proposition.
\begin{prop}\label{prop-key-transfer}
If $(k,a)=1$, then for any rational function $F(x)$ that is meaningful when modulo $1-ux^a$, we have
$$\CT_x \underline{\frac{1}{1-u x^a}} F(x) =  \CT_x \underline{\frac{1}{1-u^{1/k}x^a}} F(x^k).$$
\end{prop}
\begin{proof}
By utilizing the properties of roots of unity, it is straightforward to derive that
\begin{align*}
  \CT_x \frac{1}{\underline{1-u x^a}} F(x) &= \frac{1}{a} \sum_{i=0}^{a-1}  F(u^{-1/a} \xi^i).
  \end{align*}
Substituting $F(x)$ with $G(x)=F(x^k)$ and $u$ with $\bar{u}=u^{1/k}$ in the above equality yields
\begin{align*}
\CT_x \underline{\frac{1}{1-u^{1/k}x^a}} F(x^k) &= \frac{1}{a} \sum_{i=0}^{a-1}  G(\bar{u}^{-1/a} \xi^i) = \frac{1}{a} \sum_{ki=0}^{a-1}  F(u^{-1/a} \xi^{ki}).
\end{align*}
The proposition thus follows.
\end{proof}
When applying Proposition \ref{prop-key-transfer}, $1\leq k \leq a-1$ will be referred to as a \emph{multiplier}, and is called \emph{valid} if $(k,a)=1$. Note that $k$ and $a-k$ play equivalent roles, so we focus on $k$ with $1\leq k\leq \frac{a}{2}$.

We employ the notation
$$ [b]_a := \min\left(\mrem(b,a), a - \mrem(b,a)\right)$$
to denote the absolute value of the signed remainder of $b$ when divided by $a$.

The following result parallels Minkovski's first theorem.
\begin{prop}\label{p-existence}
Suppose $a_1,\dots, a_n$, and $a$ are positive integers. Then there exists at least one multiplier $1\leq k\leq a/2$ such that
 $$ [k a_j]_a  \leq \left\lfloor \frac{a}{\lfloor (a-1)^ \frac{1}{n}\rfloor } \right\rfloor \quad \textrm{ for all } j.$$
Such $k$ are referred to as \emph{good} multipliers.
\end{prop}

\begin{proof}
We prove this proposition by invoking the well-known pigeonhole principle. For brevity, let $\alpha = \lfloor (a-1)^{\frac{1}{n}} \rfloor$ and $\beta = \lceil \frac{a}{\alpha} \rceil$. Consider the set $S_0 = \{0, 1, \dots, a-1\}$. We form $\alpha$ boxes by partitioning $S_0$ into 
 disjoint union of $B_i = \{k \in \mathbb{Z} : (i-1)\beta \leq k < i\beta\}$ for $i = 1, \dots, \alpha$. Then any two numbers within the same box differ by at most $\beta - 1 = \left\lfloor \frac{a}{\lfloor (a-1)^{\frac{1}{n}}\rfloor } \right\rfloor$.
 
We commence by placing the $a$ numbers $k a_1 \pmod a$ for all $k \in S_0$ into the $\alpha$ boxes. There exists a box containing at least $\lceil \frac{a}{\alpha} \rceil > (a-1)^{\frac{n-1}{n}}$ elements. Let us assume that these elements are produced by $k \in S_1$ with $|S_1| > (a-1)^{\frac{n-1}{n}}$.
Next, we distribute the $|S_1|$ numbers $k a_2 \pmod a$ for all $k \in S_1$ into the $\alpha$ boxes. This process repeats, resulting in a sequence of sets $S_0 \supseteq S_1 \supseteq S_2 \supseteq \cdots \supseteq S_n$ such that $|S_i| > (a-1)^{\frac{n-i}{n}}$ for each $i$. Specifically, $S_n$ contains at least two elements, say $k_1 > k_2$. By our construction, $-\beta \leq k_1 a_i - k_2 a_i \pmod a \leq \beta - 1$ holds for all $i$. Consequently, either $k_1 - k_2$ or $a - (k_1 - k_2)$ serves as the desired good multiplier.
\end{proof}

\medskip
\emph{Identifying a nearly good multiplier.}
For $a \leq 2^n$, the definition of a good multiplier in Proposition \ref{p-existence} must be adapted since $[k a_j]_a \leq \frac{a}{2}$ holds for any $k$, and $a^{\frac{n-1}{n}} = \frac{a}{a^{\frac{1}{n}}} \geq \frac{a}{2}$.

In the case of large $a$, locating a good multiplier is challenging. This task can be rephrased as the following problem:
Find a $k$ and integers $m_j$, such that $\max_{j} |ka_j - m_j a|$ is minimized. In the language of lattices, this is akin to seeking the smallest nonzero distance (measured using the infinite norm) between two lattices: one being the $\mathbb{Z}$ span of the single vector $(a_1, a_2, \dots, a_n)$, and the other being $a \mathbb{Z}^n$.

We can approximate good multipliers using the LLL algorithm. This approach is akin to how \texttt{LattE} identifies the shortest vector. We construct the matrix
$$L = \begin{pmatrix}
                                  0.01 & a_1 &  \cdots & a_n \\
                                  \textbf{0}&  & a E_n &
                                  \end{pmatrix}$$
where $E_n$ represents the $n \times n$ identity matrix. By applying the LLL algorithm to the row vectors of $L$, the resulting matrix will have short rows of the form $(0.01 k, ka_1 - m_1a, \dots, ka_n - m_ka)$. The rows are short only when $ka_j - m_ja$ are small for all $j$. Consequently, we obtain $n$ nearly good multipliers, denoted as $k_1, k_2, \dots, k_n$. It is worth noting that the constant $0.01$ can be replaced by any other small number to give more choices.

\emph{Selection of valid multipliers}
The nearly good multipliers identified may not necessarily be valid for the application of Proposition \ref{prop-key-transfer}. Therefore, we will retain only those $k_i$ that are valid and, if necessary, include $1$ as a natural candidate. The nearly good valid multiplier is then chosen from among these candidates.
\begin{rem}
The number of valid multipliers for a given $a$ is determined by Euler's totient function $\phi(a) = a \prod_{p} (1 - p^{-1})$, where the product ranges over all the prime factors of $a$. We consider only half of these valid multipliers. In certain special cases, such as $a = 6$, the only valid multiplier is $1$.
\end{rem}

\subsection{Algorithm \texttt{DecDenu}} \label{subsec_DecDenu}
In this section, we consider the function
$$F(\lambda)=\frac{\lambda^{-a_0}}{(1-y_1\lambda^{a_1})\cdots (1-y_n \lambda^{a_n})}.$$
The approach naturally extends when the numerator is replaced by a Laurent polynomial. Replacing each $y_i$ with $u_i$ yields the function presented in Proposition \ref{prop-PFD-fund}.

We first provide the pseudo code and then explain.

\begin{algor}[DecDenu] \label{alg_DecDenu}
\mbox{  } \\
\textbf{Input}: A sequence of positive integers $(a_0,a_1, \dots,a_n)$ with $\gcd(a_1,\dots, a_n)=1$ and a number $s\in \{1,2,\dots, n\}$.
\\
\textbf{Output}: A representation of $\A_{1-y_s\lambda^{a_s}} F(\lambda)$ as in  \eqref{sec1_sum_of_function}.

\begin{enumerate}
  \item If $a_s=1$ then return $F(\lambda)(1-y_s\lambda)\big|_{\lambda=y_s^{-1}}$. Otherwise, proceed recursively as follows.

\item (new) Identify a good valid multiplier $\m$ using the LLL-algorithm. Applying the key transformation (Proposition \ref{prop-key-transfer}) yields
$$\A_{1-y_s\lambda^{a_s}} F(\lambda)=\CT_{\lambda} \underline{\frac{1}{1-y_s^{\frac 1 \m }\lambda^{a_s}}} \frac{L(\lambda^{\m})} {\prod_{j=1,j\neq s}^n (1-y_j\lambda^{\m a_j})}.$$

  \item Utilizing \eqref{e-func-mod}, we can express
$$ \A_{1-y_s\lambda^{a_s}} F(\lambda) = \A_{1-y_s^{\frac 1 \m}\lambda^{a_s}} \bar{F}(\lambda)=\CT_{\lambda}\underline{\frac{1}{1-y_s^{\frac 1 \m }\lambda^{a_s}}} \frac{\bar{L}(\lambda)} {\prod_{j=1,j\neq s}^n (1-\bar{y}_j\lambda^{b_j})},$$
where $b_j = [\m a_j]_{a_s}$ for all $j\neq s$, $\bar{F}(\lambda)$ is proper in $\lambda$ with $\bar{F}(0)=0$.

    \item According to \cite[Proposition 11]{xin2015euclid}, we derive a recurrence relation of the form
    \begin{equation}\label{eq_recs}
      \A_{1-y_s\lambda^{a_s}} F(\lambda) = \sum_{j: j\neq s, b_j>0} \A_{1-\bar{y}_j\lambda^{b_j}} \bar{F}(\lambda).
    \end{equation}
\end{enumerate}
\end{algor}

The aforementioned algorithm, except Step 2, appears as the core steps of the \texttt{CTEuclid} algorithm. However, it is Step 2 that enables the possibility of a polynomial-time algorithm.

\begin{exam}\label{exa-7}
Compute the following constant term:
\begin{align*}
    \CT_{\lambda} \frac {\lambda^{-15}}{ \underline{( 1-{\lambda}^{7}y_1)}
   (1 -{\lambda}^{2}y_{{2}} )    (1 -{\lambda}^{3}y_3 )}.
\end{align*}
\end{exam}
\begin{proof}[Solution]
Applying Proposition \ref{prop-key-transfer} with respect to $k =2 $, the constant term becomes
 \begin{align*}
   \CT_{\lambda} \frac {\lambda^{-30}}{ \underline{( 1-{\lambda}^{7}y^{1/2}_{1} )}
  (1 -{\lambda}^{4}y_{{2}} )  (1 -{\lambda}^{6}y_{{3}} ) } &=
  \CT_{\lambda} \frac {\lambda^2 y_1^3 y_2^{-1} y_3^{-1}}{ \underline{( 1-{\lambda}^{7}y^{1/2}_{1} )}
  (1 -\lambda^{3} y_1^{1/2} y_2^{-1} )  (1 -\lambda y_1^{1/2} y_3^{-1} ) } \\
   =&  \CT_{\lambda} \frac {-\lambda^2 y_1^3 y_2^{-1} y_3^{-1}}{ ( 1-{\lambda}^{7}y^{1/2}_{1} )
  \underline{(1 -\lambda^{3} y_1^{1/2} y_2^{-1} )} (1 -\lambda y_1^{1/2} y_3^{-1} )}\\
  &+  \CT_{\lambda} \frac {-\lambda^2 y_1^3 y_2^{-1} y_3^{-1}}{ ( 1-{\lambda}^{7}y^{1/2}_{1} )
  (1 -\lambda^{3} y_1^{1/2} y_2^{-1} ) \underline{(1 -\lambda y_1^{1/2} y_3^{-1} )}},
 \end{align*}
For the first term, we have
  \begin{align*}
     \CT_{\lambda}& \frac {-\lambda^2 y_1^3 y_2^{-1} y_3^{-1}}{ ( 1-{\lambda}^{7}y^{1/2}_{1} )
  \underline{(1 -\lambda^{3} y_1^{1/2} y_2^{-1} )}  (1 -\lambda y_1^{1/2} y_3^{-1} ) } \\
   &= \CT_{\lambda} \frac {-\lambda^2 y^3_1 y_2^{-1} y_3^{-1} }{ (1-  \lambda y^2_2 y_1^{-1/2})
  \underline{(1 -\lambda^{3} y_1^{1/2} y_2^{-1} )} (1 -\lambda y_1^{1/2} y_3^{-1} )  } \\
    &=\CT_{\lambda} \frac {\lambda^2 y^3_1 y_2^{-1} y_3^{-1} }{\underline{ (1-  \lambda y^2_2 y_1^{-1/2})}
  (1 -\lambda^{3} y_1^{1/2} y_2^{-1} ) \underline{(1 -\lambda y_1^{1/2} y_3^{-1} )}  } \\
    &= \frac{y_1^4 y_2^{-5} y_3^{-1} }{(1- y_1^2 y_2^{-7}) (1- y_1 y_2^{-2} y_3^{-1}) }  + \frac{y_1^2 y_2^{-1} y_3 }{(1-y_3 y_2^2 y_1^{-1})(1-y_3^3 y_1^{-1} y_2^{-1})};
  \end{align*}
  For the second term, we have
  \begin{align*}
    \CT_{\lambda} \frac {-\lambda^2 y_1^3 y_2^{-1} y_3^{-1}}{ ( 1-{\lambda}^{7}y^{1/2}_{1} )
   (1 -\lambda^{3} y_1^{1/2} y_2^{-1} ) \underline{(1 -\lambda y_1^{1/2} y_3^{-1} )} }
   = \frac{-y_3 y_1^2  y_2^{-1}}{ (1- y^7_3 y_1^{-3}) (1- y_3^3 y_1^{-1} y_2^{-1})}.
  \end{align*}
 In total, we obtain three output terms. It is worth noting that
 we will obtain four terms if we use the \texttt{CTEuclid} algorithm (with the default value $k=1$).
\end{proof}

Notably, the algorithm incorporates fractional powers of the $y$ variables when Step 2 is employed, as displayed
in the middle steps of Example \ref{exa-7}. On the other hand, the three output terms contain only integer powers of the $y$ variables.
This is not a coincidence. Indeed, we have the following result, which is pivotal when considering constant terms involving multiple variables:
\begin{prop}
The output of \texttt{DecDenu} only contains integer powers of the $y$ variables.
\end{prop}
\begin{proof}
We encode an Elliott rational function by a matrix by the following correspondence:
$$\frac{y^{\alpha_0}\lambda^{a_0}}{\prod_{i=1}^n (1-y^{\alpha_i} \lambda^{a_i})}   \mapsto  \begin{pmatrix}
  a_1 & a_2 &\cdots & a_n & a_0 \\
  \alpha_1 & \alpha_2 &\cdots & \alpha_n & \alpha_0
\end{pmatrix}.
$$
Subsequently, \texttt{DecDenu} can be applied to these matrices.

Step 2 corresponds to the transformation $\gamma_{s,\m}$ as defined in Definition \ref{def-tramat} below, and Step 3 corresponds to a sequence of operations $\epsilon_{s, l, h}$ for appropriate choices of $l$ and $h$ as per the same definition.

Each term in the output corresponds to a matrix whose first row consists entirely of zeros, except for a single $1$ in the first $n$ columns. By permuting the columns, we can assume that the sole $1$ is located in the first column without loss of generality. The proposition then follows from Corollary \ref{cor-int} below.
\end{proof}

\begin{defn}\label{def-tramat}
Define $\gamma_{s, m}$ as the transformation that modifies matrix $M$ by scaling its first row by a factor of $m$, and scaling its $s$-th column by a factor of $\frac{1}{m}$. Additionally, let $\epsilon_{s, l, h}$ denote the column operation that adds $h$ times the $s$-th column to the $l$-th column of $M$.
\end{defn}

\begin{theo}\label{theo-intmat}
Suppose $M$ is a matrix, and $\rho$ is either $\gamma_{s, m}$ or $\epsilon_{s, l, h}$, where $m$ and $h$ are integers. If all $k$-th order minors of $M$ that include the first row are integers, then the corresponding $k$-th order minors in $\rho(M)$ are also integers.
\end{theo}
\begin{proof}
Consider matrix $M = (a_{ij})$, and let $M[I,J]$ be the submatrix of $M$ with row set $I=\{1,i_2,\ldots,i_k\}$ and column set $J=\{j_1,j_2,\ldots,j_k\}$.

Case $s \in J$: it is evident that for both $\gamma_{s, m}$ and $\epsilon_{s, l, h}$, the determinant $\det(\rho(M)[I,J])$ remains unchanged and equals $\det(M[I,J])$.

Case $s \not\in J$: for $\rho=\gamma_{s, m}$, $\rho(M)[I,J]$ is obtained from $M[I,J]$ by scaling the first row by $m$, resulting in $\det(\rho(M)[I,J])=m \det(M[I,J])$. For $\rho=\epsilon_{s, l, h}$, if $l \notin J$, then $\rho(M)[I,J]$ is identical to $M[I,J]$; if $l \in J$, then we have
\begin{align*}
\rho(M)[I,J] =
\begin{pmatrix}
 a_{1,j_1} & \cdots & h a_{1,s} +  a_{1,l} & \cdots &  a_{1,j_k} \\
a_{i_2,j_1} & \cdots & h a_{i_2,s} + a_{i_2,l} & \cdots & a_{i_2,j_k} \\
\vdots & \ddots & \vdots & \ddots & \vdots \\
a_{i_k,j_1} & \cdots & h a_{i_k,s} + a_{i_k,l} & \cdots & a_{i_k,j_k}
\end{pmatrix}.
\end{align*}
By splitting the column $l$ in $\rho(M)[I,J]$, we obtain $\det(\rho(M)[I,J]) = h \det(M[I,J']) +  \det(M[I,J])$, where $J'$ is obtained by replacing $l$ with $s$ in $J$. Since $\det(M[I,J'])$ and $\det(M[I,J])$ are both integers by assumption, the theorem is established.
\end{proof}

\begin{cor}\label{cor-int}
Let $M$ be an integer matrix with the gcd of its first row being $1$. If $M$ undergoes a sequence  of transformations $\rho_i$, each
being $\gamma_{s, m}$ or $\epsilon_{s, l, h}$, to yield the matrix
\begin{equation} \label{eq-transformedMat}
 \begin{pmatrix}
1 & 0 \\
\alpha & A
\end{pmatrix},
\end{equation}
then $A$ is also an integer matrix.
\end{cor}
\begin{proof}
Given that $M$ is integral, the $2$-order minors involving the first row in $M$ are integers. This property is preserved through
any transformation $\rho_i$ by Theorem \ref{theo-intmat}.
Therefore, the $2$-order minors involving the first row in the transformed matrix \eqref{eq-transformedMat}
are also integers, thus confirming the corollary.
\end{proof}

\subsection{Complexity Analysis}
\def\TB{T_\texttt{B}}
\def\TD{T_\texttt{D}}
In order to do complexity analysis, it is better to associate the algorithm
DecDCone a rooted tree $\TD$.

The root of $\TD$ is labeled by $(\ci{k};a_1',\dots, a_n')$, where $a_k'=a_k$ and $0\leq a_i'=[a_i]_{a_k}\leq a_k/2 $ for all $i\neq k$.
It suffices to construct the children of a general node $v$,
whose label is denoted $(\ci{s};v_1,\dots, v_n)$, where $0\leq v_i \leq v_s/2$ for all $i\neq s$, and $\gcd(v_1,\dots, v_n)=1$.
The index of $v$ is $\ind(v)=v_s$, and the dimension $\dim(v)$ of $v$ is the number of nonzero entries of $v$ minus $1$.
i) If $v_s=1$, then $v$ is of dimension $0$ and is called a leaf; ii) If $v_s>1$, then $\dim(v)\geq 1$ by $\gcd(v_1,\dots, v_n)=1$.
According to the recursion step 2.3.3 in Algorithm \texttt{DecDenu}, $v$ is an internal node with $\deg(v)=\dim(v)$ children, each of the form
$(\ci{j};b_1,\dots, b_n)$ of index $b_j$. It is possible that two nodes $v$ and $v'$ have the same label, and hence the two subtrees $T_v$ and $T_{v'}$
are isomorphic according to our construction, however, their corresponding rational functions are different.

It should be clear that the number of terms in the output of DecDCone is exactly the number $nl(\TD)$ of leaves of $\TD$. It is an important quantity to measure
the computation time of DecDcone. Further observe that we need to find a good valid multiplier $\m$ for each internal node of $\TD$. An internal node
is called an \emph{LLL-node} if LLL's algorithm is used to find $\m$. We will see that only a small portion of internal nodes are LLL-nodes.

It is hard to estimate $nl(\TD)$ for general $\TD$. As each node $v$ of $\TD$ has $\dim(v)\leq d$ children, it is natural to consider the \emph{depth} of $\TD$.
The depth $dep(v)$ of a node $v$ in $\TD$ is the number of edges in the unique path from $v$ to the root, and the depth of $\TD$ is
defined by $dep(\TD)=\max_{v\in T} \ dep(v)$.
Clearly, if $dep(\TD)=\ell$, then $nl(\TD)\leq d^\ell$, with equality holds only when $\TD=T^{d,\ell}$, the unique tree whose internal nodes have exactly $d$ children
and whose leaves have depth $\ell$. We call $T^{d,\ell}$ the complete $d$-ary rooted tree with a uniform depth $\ell$, i.e., each leaf has depth $\ell$.

However, it is hard to estimate the depth of $\TD$ either. Let $\bar{\TD}$ be similar to $\TD$ but ignore the coprime condition. Then we can show that $nl(\bar{\TD})$
is polynomially bounded. This is very similar to Barvinok's decomposition tree, so let us also construct the tree $\TB$ for Barvinok's decomposition.
We will see that the bound for $nl(\bar{\TD})$ is much smaller than the bound for $nl(\TB)$.

Let us estimate the depth $\ell=dep(\TB)$ for $\TB$ first. For any path $\K=\K_0\to \K_1 \to\cdots \to \K_\ell$ from the root to a leaf in $\TB$, we have
the inequality $\ind(\K_{i+1})\leq \ind(\K_i)^{\frac{d-1}{d}}$ for $0\leq i \leq \ell-1$. It follows that $\ind(\K_\ell) \leq \ind(\K_0)^{(\frac{d-1}{d})^\ell}$. If the right hand side is less than $2$ then $\K_\ell$ must be unimodular. This gives an upper bound of $dep(\TB)$ by
$$dep(\TB)\leq \left\lfloor 1+ \frac{\log\log\ind(\K_0)}{\log d -\log(d-1)} \right\rfloor,$$
and hence gives a polynomial bound in $\log\ind(\K_0)$ for $nl(\TB)$.

The depth $\ell=dep(\bar{\TD})$ for $\bar{\TD}$ is estimated in a similar way.
For any path $D_0\to D_1\to \cdots \to D_\ell$ from the root to a leaf in $\bar{\TD}$, where $D_i$ corresponds to a denumerant cone,
we have the inequality
$\ind(D_{i+1})\leq  \left\lfloor \frac{\ind(D_i)}{\lfloor (\ind(D_i)-1)^ \frac{1}{d}\rfloor } \right\rfloor$
  by Proposition \ref{p-existence}.
This is very similar to that for $\TB$, but we also have $\ind(D_{i+1})\leq \ind(D_i)/2$, which is better when $\ind(D_i)\leq 2^d$.
Thus $dep(\bar{\TD})$ is roughly bounded by
$$d+ \frac{\log\log\ind(\D_0)-\log \log(2^d)}{\log d- \log(d-1)}.$$

At the very beginning of this work (around 2011), we used $\ind(\K_0)=a_s^{d-1}$ under the assumption that $(a_s,a_i)=1$ for all $i\neq s$. Then
the bound of $dep(\TB)$ is huge when compared with the bound of $dep(\TD)$. This, together with computer data,
convinced us that DecDCone is much better than \texttt{LattE}, even though we are not able to show its polynomial complexity.
  Later, we learn that
LattE used polar trick, so that $\ind(\K)$ is replaced with $\ind(\K^*)$ which is usually just $a_s$, see \cite{baldoni2013coefficients}.
Then the bound for $dep(\bar{\TD})$ is less than the bound for $dep(\TB)$ by roughly $\frac{\log d}{ \log d -\log(d-1)} -d$, which is positive when
 $d\geq 5$ and is greater than $d$ if $d\ge 9$.

The above seems evident that DecDCone might be better than \texttt{LattE} even if $\TD$ has some \emph{bad} nodes, by which we mean that at this node, 
the candidates $\m$ found by LLL's algorithm or $\m=1$ can not reduce the index $a_s$ to be smaller than $C a_s^{\frac{d-1}{d}}$ for a suitable constant $C$. 
However, \texttt{LattE} performs much better than the estimated bound for denumerant cones.
For instance, computer experiments support that $nl(\TD)=nl(\TB)$ if we start with the denumerant cone of index $\leq 5$. In particular, when the index $a_s=3$, $dep(\TB)$ is always $1$ and
$nl(\TB)$ is always $d$. This means that the decrease from index $3$ to $2$ never happens, while the suggested bound is $3^{(d-1)/d}$, which is $>2.498$ if $d\geq 6$.

Bad nodes in $\TD$ cannot be avoided so $nl(\TD)$ is hard to estimate. For instance, consider the node $v=(\ci{1}; 2a, v_2, v_3,\dots, v_n)$ with $v_2=a$. Then $[\m a_2]_{2a}=a $ whenever $(\m,2a)=1$. This means $v$ must have a child $w$ of index $a$.
The $w$ is of the form $(\ci{2}; 2a,a,w_3,\dots,w_n)\to (\ci{2}; 0,a,w_3,\dots,w_n)$, and hence we have a dimension decrease $\dim w\leq \dim v-1$.

Though $\TD$ has bad nodes while $\TB$ has not, $\TD$ has some nicer properties than $\TB$.
\begin{enumerate}
\item In $\TD$, a node of index $a_s\leq 2^d$ has children of indices at most $a_s/2$, while in $\TB$ a node of index $\ind(\K)$ has children of indices at most
$\ind(\K)^{\frac{d-1}{d}}$, which is not less than $\ind(\K)/2$ when $\ind(\K)\leq 2^d$.

\item Each node in $\TB$ corresponds to a $d$ dimensional simplicial cone, while each node $v$ in $\TD$ corresponds to a denumerant cone of dimension $\dim v$.

\item If $w$ is a child of $v$ in $\TD$, then $\dim w\leq \dim v$, while all nodes in $\TB$ have dimension $n-1$.
\end{enumerate}

In $\TB$, the decrease of the dimension happens when new $0$'s appear.
If the index is reduced to be smaller than $2d$,
 then the decrease of the dimension is guaranteed by the pigeonhole principle.

Now we turn to the second important quantity of DecDCone: the number $nl^3(T)$ of LLL-nodes in $T$. This is because LLL's algorithm in a $d$ dimensional lattice has an average complexity $O(d^3\log(d))$ \cite{cl2013LLL}. In all the other steps, only simple arithmetic operations are performed.
This is very different from the $\TB$ for \texttt{LattE}.

\begin{enumerate}
\item All internal nodes in $\TB$ are LLL-nodes while there is only a small portion of internal nodes in $\TD$ are LLL-nodes.

\item Matrix inverse is needed for each internal node of $\TB$, but is never needed in $\TD$.
\end{enumerate}
The second item is obvious as described in Section \ref{sec_BarAlg} and Subsection \ref{subsec_DecDenu}. For the the first item, we use Proposition \ref{prop_less13}  in the next subsection,
which says that only nodes of indices $\geq 14$ are LLL-nodes.
Since every leaf has an index $1$, most nodes close to a leaf are not LLL-nodes. More precisely, for a node of index $\ge 14$, its children
can hardly be all leaves. Take $T^{d,\ell}$ as an example. It has $d^0+\cdots+d^{\ell-1}$ internal nodes, but nodes of depth $\ell-1$ and $\ell-2$ can hardly be LLL-nodes. Thus the number of LLL-nodes is probably
less than $d^0+\cdots+d^{\ell-3}$.

The above argument is supported by computer experiments in Subsection \ref{subsec_CompExp}. Table \ref{tab-highdim} records some random examples with dimensions from 8 to 14. The computer reports that only about $4\%$ of the internal nodes are LLL-nodes. As a result,
even if $nl(\TD)>nl(\TB) $, the computation time for DecDCone may still be less than that of \texttt{LattE}. See Table \ref{table_result_cuww}.

\subsection{The optimal  valid multiplier when $a_s<14$ }
In order to have a better understanding of $nl(\TD)$, we introduce the rooted tree $\TD^1$, which is similar to $\TD$
except that we choose $\m=1$ at each internal node. This corresponds to the \texttt{CTEuclid} algorithm \cite{xin2015euclid}.
Suppose the roots of $\TD^1$ and $\TD$ both have label $(\ci{s};a_1,\dots,a_n)$.

 Denote by $f^1(\ci{s};a_1,\dots,a_n)=nl(\TD^1)$, $f^{m}(\ci s;a_1,\dots,a_n)=nl(\TD)$.
The superscripts correspond to our choice of $\m$ at each internal node. Then the $f^h$ (with $h= 1$ or $m$)  is recursively determined by the following rules:

\begin{enumerate}
\item $f^h(\ci{s};a_1,\dots, a_n)=f^h(\ci 1;a_s,b_1,\dots, b_{n-1})$ where the $b$'s can be any permutation of $a_1,\dots, a_{s-1}, a_{s+1},\dots, a_n$.

\item Zero entries can be omitted, and the initial value is $f^h(\ci{1}; 1)=1$.

\item Find a multiplier $\m$ according to $h$: $\m=1$ if $h=1$; $\m$ is the optimal valid multiplier if $h=m$. Then
 $$f^h(\ci s; a_1,a_2,\dots,a_n)=f^h(\ci s; b_1,b_2,\dots,b_n)$$
 where $b_s=a_s$ and $b_i=[\m a_i]_{a_s}$ if $i\neq s$.

  \item Now $0\leq b_j=[\m a_j]_{a_s} \leq a_s / 2$ for all $j \neq s$ as above. We have
        $$f^h(\ci s; a_1,a_2,\dots,a_n) = \sum_{j\neq s, b_j\neq 0} f^h(\ci j; b_1,b_2,\dots,b_n).$$
\end{enumerate}
Observe that the recursion occurs at the fourth rule, which turns a term of index $a_s$ into a sum of
at most $d=n-1$ terms of smaller indices.

For $f^1$ we only have the obvious bound $f^1(\ci s; a_1, a_2, \dots, a_n) < (n-1)^{\log a_s}$, since in each recursion, the index $a_s$
is reduced by half. Computer experiment suggests that $f^1(\ci s; a_1, a_2, \dots, a_n)$ is not polynomial.

\begin{exam}\label{ex-12}
Let us revisit Example \ref{exa-7} but focus on the number of the output terms. This will illustrate the distinction between $f^m$ and $f^1$.
By using $\m=2$, we obtain
\begin{eqnarray*}
  f^m(\ci 1;7,2,3) &\overset{\m = 2}{=}& f^m(\ci 1;7,3,1)\\ %
   &=&f^{m}(\ci{2};7,3,1)+f^{m}(\ci{3};7,3,1)\\
   &=&f^{m}(\ci{2};1,3,1)+1 =2+1=3,
\end{eqnarray*}
which agrees with the solution for Example \ref{exa-7}. While using $\m=1$ gives
\begin{eqnarray*}
  f^1(\ci 1;7,2,3) &=& f^1(\ci 2;7,2,3)+ f^1(\ci 3;7,2,3) \\
   &=& f^1(\ci 2;1,2,1)+ f^1(\ci 3;1,1,3)\\
   &=&2+2=4.
\end{eqnarray*}
\end{exam}
Now we give explicit formulas
of $f^1(\ci 1;a_1,\dots, a_n)$ for $a_1\leq 13$,
and use them to find the optimal valid multiplier $\m$ for $f^{m}(\ci 1;a_1,\dots, a_n)$ when $a_1\leq 13$.
When $a_1$ is fixed, we encode $(a_2,\dots,a_n)$ as
$(1^{r_1}, \dots, k^{r_k})$, where $k=\lfloor a_1/2\rfloor$, meaning that there are $r_1$ $1$'s, $r_2$ $2$'s, and so on.
The recursion can be rewritten as
$$ f^1(\ci 1;a_1,1^{r_1}, \dots, k^{r_k})=\sum_{i=1}^k r_i f^1(\ci 1; i, a_1, 1^{r_1},\dots, (i-1)^{r_{i-1}},(i+1)^{r_{i+1}},\dots, k^{r_{k}}).$$

The following formulas for $a_1<14$ are easy to compute.
\begin{align*}
 &f^1(\ci 1;2,1^{r_1}) =f(1;3,1^{r_1})=r_1 ,\\
 &f^1(\ci1;4, 1^{r_1}, 2^{r_2}) =r_1 r_2+r_1 ,\\
 &f^1(\ci1;5, 1^{r_1}, 2^{r_2})= r_1 r_2+r_1+r_2, \\
 &f^1(\ci1;6, 1^{r_1}, 2^{r_2}, 3^{r_3})=r_1 (r_2+r_3+1)+2r_2 r_3, \\
 &f^1(\ci1;7, 1^{r_1}, 2^{r_2}, 3^{r_3}) =r_1 (r_2+r_3+1)+(2 r_3+1) r_2+r_3, \\
 &f^1(\ci1;8, 1^{r_1}, 2^{r_2}, 3^{r_3}, 4^{r_4})=(r_4+1) (r_2+1) (r_1+r_3)+(r_4+r_2+r_1) r_3 ,\\
 &f^1(\ci1;9, 1^{r_1}, 2^{r_2}, 3^{r_3}, 4^{r_4})=(r_1+1+r_3) (r_2 r_4+r_2+r_4)+r_1+(r_4+r_2+r_1) r_3,  \\
 & f^1(\ci1;10, 1^{r_1},  2^{r_2},  \dots, 5^{r_5}) = r_1 ((r_2+2) r_4+(r_2+r_3+1) (r_5+1))+3 r_4 (r_3+r_5) \\
                                               & \qquad +(r_4+2) (r_5+r_2) r_3+2 r_5 (r_4+1) r_2 +r_3, \\
 &f^1(\ci1;11, 1^{r_1},  2^{r_2}, \dots, 5^{r_5})=r_1+(r_4r_2+r_2 +r_4) (r_1+1+r_3+r_5)\\
                    & \qquad+r_3 (r_1+r_2+1+r_4+r_5)+r_5 ((r_1+r_4+2) (r_2+r_3+1)-1) ,\\
 &f^1(\ci1;12, 1^{r_1},2^{r_2}, \dots, 6^{r_6})=((r_6+1) (r_2+r_3+1)+(r_2+2 r_6+1) r_4) r_1 \\
                     & \qquad+((r_4+2 r_6+2) r_2+r_4 (3 r_6+2)) r_3 \\
                        & \qquad+ ((r_2+r_3+2) r_1+(r_4+2 r_6+2) (r_2+r_3) +r_4 r_2+ (2 r_6+3) r_4+3 r_6+1) r_5, \\
&f^1(\ci1;13, 1^{r_1},2^{r_2}, \dots, 6^{r_6})=r_1+r_2 (r_1+1+r_3+r_5)+r_3 (r_1+r_2+1+r_4+r_5)\\
                 & \qquad+r_4 (r_1+1+r_3+r_5+(r_2+r_6) (r_1+1+r_3+r_5))\\
                & \qquad+r_5 (r_1+r_4+r_6+(r_2+r_3+1) (r_1+r_4+r_6+1))\\
               & \qquad+r_6 (r_1+r_5+1+(r_2+r_4) (r_1+1+r_3+r_5)+r_3 (r_1+r_2+1+r_4+r_5)).
\end{align*}

\begin{lem}\label{lem_leq6}
If $a_1\leq 6$, then we have $f^m(\ci1;a_1,1^{r_1}, \dots, k^{r_k})=f^1(\ci1;a_1,1^{r_1}, \dots, k^{r_k})$.
\end{lem}
\begin{proof}
When $a_1 \leq 4$ or $a_1=6$, the only valid choice for $m$ is $1$, thus $f^m$ reduces to $f^1$.

When $a_1=5$, the valid multipliers are $1$ and $2$. Utilizing $1$ yields $f^1(\ci1;5, 1^{r_1}, 2^{r_2})$, while employing $2$ results in $f^1(\ci1;5, 2^{r_1}, 1^{r_2}) = f^1(\ci1;5, 1^{r_1}, 2^{r_2})$ according to the explicit formula.
\end{proof}

\begin{prop}\label{prop_less13}
If $a_1 \leq 13$, then the function $f^m(\ci1;a_1,1^{r_1}, \dots, k^{r_k})$ can be computed efficiently, with the optimal choice of $\m$ determined through a few comparisons.
\end{prop}

\begin{proof}
Given that $a_1 \leq 13$, it follows that any choice of the multiplier $\m$ will result in an index that does not exceed 6. By Lemma \ref{lem_leq6}, we only need to make the initial choice of $\m$, and there are at most 6 possible choices for $\m$. Let $g_\m$ denote the value obtained by employing the multiplier $\m$. We shall select the optimal multiplier $\m_0$ to be $\m_0=i$ if $g_i$ is the smallest among all valid multipliers $\m$.

For some specific values of $a_1$, we can derive simple conditions. We illustrate the case $a_1 = 7$ in detail, as the other cases follow a similar pattern. In this case, the optimal multiplier $\m_0$ must be chosen from the set $\{1,2,3\}$. We have:
\begin{align*}
g_1 &= f^1(\ci1;7, 1^{r_1}, 2^{r_2},3^{r_3}), \\
g_2 &= f^1(\ci1;7, 2^{r_1}, 4^{r_2},6^{r_3}) = f^1(\ci1;7, 1^{r_3},2^{r_1},3^{r_2}), \\
g_3 &= f^1(\ci1;7, 3^{r_1}, 6^{r_2},9^{r_3}) = f^1(\ci1;7, 1^{r_2},2^{r_3},3^{r_1}).
\end{align*}
Thus, by the explicit formula of $f^1(\ci1;7, 1^{r_1}, 2^{r_2},3^{r_3})$,  we obtain:
\begin{align*}
g_1 - g_3 &= f^1(\ci1;7, 1^{r_1}, 2^{r_2},3^{r_3}) - f^1(\ci1;7, 1^{r_3},2^{r_1},3^{r_2}) = -(r_1 - r_3)r_2, \\
g_1 - g_2 &= f^1(\ci1;7, 1^{r_1}, 2^{r_2},3^{r_3}) - f^1(\ci1;7, 1^{r_2},2^{r_3},3^{r_1}) = -(r_1 - r_2)r_3.
\end{align*}
If both of these differences are negative, i.e., $r_1 = \max(r_1, r_2, r_3)$, then the optimal multiplier $\m_0$ is $1$. The conditions for $\m_0$ being $2$ or $3$ can be similarly derived.
In summary, assuming $r_i = \max(r_1, r_2, r_3)$, then if $i=1$ then $\m_0=1$; if $i=2$ then $\m_0=3$; if $i=3$ then $\m_0=2$.

For $a_1 = 8$, the optimal $\m_0$ must be chosen from $\{1,3\}$. We shall choose $\m_0 = 3$ only when $r_1 < r_3$.

For $a_1 = 9$, the optimal $\m_0$ must be chosen from $\{1,2,4\}$. Assume $r_i = \max(r_1, r_2, r_4)$. Then if $i=1$, then $\m_0 = 1$; if $i=2$ then $\m_0=4$; if $i=4$ then $\m_0=2$.

For $a_1 = 10$, the optimal $\m_0$ must be chosen from $\{1,3\}$. We shall choose $\m_0=3$ only when $( r_1 + r_2 - r_3 - r_4 ) r_5 + 2r_1 r_2 - 2r_3 r_4 < 0$.

For $a_1 = 12$, the optimal $\m_0$ must be chosen from $\{1,5\}$. We shall choose $\m_0=5$ only when $r_1 < r_5$.

For $a_1 = 11$ or $13$, we cannot derive simple conclusions as in the other cases.
\end{proof}

%
%

\begin{theo}
Suppose $a_1, a_2, a_3$ are pairwise coprime. Then $f^m(\ci i; a_1, a_2, a_3) =  \lfloor\log(a_i)\rfloor $, and consequently $d(a_0; a_1, a_2, a_3)$ can be computed in polynomial time.
\end{theo}
\begin{proof}
Consider $f^m(\ci 1; a_1, a_2, a_3)$, and analogously for $i=2,3$. Since $\gcd(a_1, a_3) = 1$, there exist $u, v$ such that $a_1 v + a_3 u = 1$ and $\gcd(u, a_1) = 1$. Thus, by setting the valid multiplier $k = u$, we have
\begin{align*}
f^m(\ci 1; a_1, a_2, a_3) \overset{k=u}=& f^m(\ci 1; a_1, b_2, 1);
\end{align*}
where $b_2 = [a_2u]_{a_1}$. The recursion for $f^1(\ci 1; a_1, b_2, 1)$ is easy derived: $f^1(\ci 1; a_1, b_2, 1) = f^1(\ci 2; b_1, b_2, 1) + 1 = \cdots = \lfloor \log(a_1) \rfloor$.
An illustrative example is provided in Example \ref{exa-7} and \ref{ex-12}. By specializing $y_i = 1$, we obtain $d(a_0; a_1, a_2, a_3)$, which can also be computed in polynomial time, as shown in \cite{XinTodd2023}.
\end{proof}

\subsection{Computer Experiments} \label{subsec_CompExp}
This section presents the outcomes of our computational experiments aimed at decomposing denumerant cones $D_1(a_1,\dots, a_n)$ of index $a_1$. We utilized two C++ packages: \texttt{LattE 1.7.5} and our own \texttt{DecDenu}. The results are multivariate rational generating functions corresponding to unimodular cones. We report the computation time and the number of terms $nl(\TD)$ and $nl(\TB)$ in the output.

In Table \ref{tab-highdim}, we list two random knapsack cones for each $9\leq n \leq 15$. Notably, there is only one instance (when $n=10$ and boldfaced) for which $nl(\TD)> nl(\TB)$.

{\tiny
\begin{table}[h]
  \centering

    \begin{tabular}{r|p{0.5cm}p{0.5cm}p{0.5cm}p{0.5cm}p{0.5cm}p{0.5cm}p{0.5cm}p{0.5cm}p{0.5cm}p{0.5cm}p{0.5cm}p{0.5cm}p{0.5cm}p{0.5cm}p{0.5cm}|r|r}
   $n$ & \multicolumn{15}{c|}{Knapsack cone $\D_1(a_1,\dots a_n)$} & $nl(\TD)$ & $nl(T_B)$ \\ \hline
    9  & \textbf{1285} & 2549 & 2209 & 2402 & 2018 & 2789 & 1181 & 2369 & 121 &    &    &    &    &    &    & 10342 & 11341 \\
    9  & \textbf{1565} & 2594 & 2882 & 2988 & 2876 & 544 & 1621 & 740 & 2372 &    &    &    &    &    &    & 11063 & 12808 \\
    10 & \textbf{422} & 1980 & 2478 & 1360 & 2179 & 1992 & 2857 & 1326 & 78 & 2421 &    &    &    &    &    & \textbf{8117} & 7701 \\
    10 & \textbf{2937} & 600 & 2895 & 538 & 584 & 2175 & 1636 & 2942 & 1905 & 509 &    &    &    &    &    & 40591 & 50123 \\
    11 & \textbf{681} & 640 & 1082 & 2115 & 2937 & 965 & 2690 & 1572 & 701 & 596 & 224 &    &    &    &    & 22747 & 25557 \\
    11 & \textbf{1576} & 2362 & 226 & 2059 & 2078 & 2694 & 1824 & 1320 & 1908 & 2968 & 1547 &    &    &    &    & 29749 & 40456 \\
    12 & \textbf{1439} & 799 & 2358 & 241 & 743 & 2370 & 2188 & 1713 & 1114 & 783 & 922 & 1124 &    &    &    & 75889 & 89577 \\
    12 & \textbf{2934} & 1928 & 2894 & 1687 & 2542 & 2633 & 662 & 2545 & 1184 & 1250 & 2357 & 1539 &    &    &    & 106726 & 155285 \\
    13 & \textbf{951} & 1249 & 1796 & 2396 & 1838 & 728 & 930 & 1266 & 196 & 2353 & 701 & 1906 & 1301 &    &    & 56259 & 71119 \\
    13 & \textbf{2097} & 1020 & 2525 & 628 & 1080 & 581 & 2709 & 1322 & 149 & 1125 & 2309 & 1210 & 1878 &    &    & 291075 & 501261 \\
    14 & \textbf{1300} & 1340 & 2934 & 1188 & 1696 & 1716 & 67 & 167 & 2390 & 950 & 1218 & 1201 & 2757 & 2584 &    & 256285 & 314738 \\
    14 & \textbf{2800} & 631 & 608 & 2136 & 2925 & 163 & 628 & 1387 & 1337 & 2370 & 2226 & 2562 & 1550 & 739 &    & 833283 & 999965 \\
    15 & \textbf{758} & 148 & 1880 & 43 & 281 & 2169 & 528 & 243 & 1589 & 1187 & 1145 & 290 & 268 & 2643 & 317 & 215849 & 428843 \\
    15 & \textbf{2841} & 438 & 775 & 2129 & 2919 & 1284 & 1374 & 613 & 2917 & 674 & 1740 & 843 & 1834 & 1314 & 924 & 1349790 & 1393263 \\\hline
    \end{tabular}%
    \caption{Knapsack cones with dimensions ranging from 8 to 14. \label{tab-highdim}}
\end{table}%
}

For a comparative analysis, we list several hard knapsack problems from \cite{aardal2002hard} in Table \ref{table_hardknapsack}. The first column provides the instance names, with the data for instances cuww1-3 omitted due to their negligible computation times. Instances cuww4-5, and prob1-10 are specifically crafted to possess the structure $a_i = Mp_i + Nr_i$ for certain integers $M$ and $N$, and for short integer vectors $p=(p_1,\dots, p_n)$ and $r=(r_1,\dots, r_n)$. The far-right column specifies the Frobenius number $F(a_1,\dots, a_n)$. It is a notoriously difficult task for many algorithms, including the standard branch and bound approach, to confirm that $d(F;a_1,\dots, a_n)=0$. Instances prob11-20 represent random examples.

{\tiny
\begin{table}
\centering
\begin{tabular}{|c|cccccccccc|c|c|} \hline
instance&&&&&a&&&& &&Frobenius \# \\ \hline
cuww4&13211 &13212 &39638 &52844 &66060 &79268 &92482 & & & &104723595 \\
cuww5&13429 &26850 &26855 &40280 &40281 &53711 &53714 &67141 & & & 45094583 \\
pro1& 25067 &49300 &49717 &62124 &87608 &88025 &113673 &119169 & & & 33367335   \\
prob2& 11948 &23330 &30635 &44197 &92754 &123389 &136951 &140745 & & &14215206  \\
prob3  &39559 &61679 &79625 &99658 &133404 &137071 &159757 &173977 & & &58424799   \\
prob4& 48709 &55893 &62177 &65919 &86271 &87692 &102881 &109765 & & &60575665   \\
prob5& 28637 &48198 &80330 &91980 &102221 &135518 &165564 &176049 & & & 62442884  \\
prob6 &20601 &40429 &40429 &45415 &53725 &61919 &64470 &69340 &78539 &95043 &22382774   \\
prob7 &18902 &26720 &34538 &34868 &49201 &49531 &65167 &66800 &84069 &137179 &27267751   \\
prob8& 17035 &45529 &48317 &48506 &86120 &100178 &112464 &115819 &125128 &129688& 21733990  \\
prob10 & 45276 &70778 &86911 &92634 &97839 &125941 &134269 &141033 &147279 &153525 &106925261   \\ \hline
prob11 &11615 &27638 &32124 &48384 &53542& 56230& 73104 &73884  &112951 &130204 &577134\\
prob12 &14770 &32480 &75923 &86053 &85747& 91772& 101240& 115403& 137390& 147371& 944183\\
prob13& 15167 &28569 &36170 &55419 &70945& 74926& 95821 &109046 &121581 &137695 &765260\\
prob14 &11828 &14253 &46209 &52042 &55987& 72649& 119704& 129334& 135589& 138360& 680230\\
prob15 &13128 &37469 &39391 &41928 &53433& 59283& 81669 &95339  &110593 &131989 &663281\\
prob16 &35113 &36869 &46647 &53560 &81518& 85287& 102780& 115459& 146791& 147097& 1109710\\
prob17 &14054 &22184 &29952 &64696 &92752& 97364& 118723& 119355& 122370& 140050& 752109\\
prob18 &20303 &26239 &33733 &47223 &55486& 93776& 119372& 136158& 136989& 148851& 783879\\
prob19 &20212 &30662 &31420 &49259 &49701& 62688& 74254 &77244  &139477 &142101 &677347\\
prob20 &32663 &41286 &44549 &45674 &95772& 111887&117611& 117763& 141840& 149740& 1037608\\\hline
\end{tabular}
\centering
\caption{Hard knapsack problems from \cite{aardal2002hard}.\label{table_hardknapsack}}
\end{table}
}
By employing \texttt{DecDenu} and \texttt{LattE}, we have re-evaluated the values of $d(F+i;a_1,\dots, a_n)$ for $i$ ranging from $0$ to $100$ for the instances listed in Table \ref{table_hardknapsack}. The reported times correspond solely to the computation of the unimodular cones. It is worth noting that the transformation from unimodular cones to $d(F+i;a_1,\dots, a_n)$ is a critical step, but the two software packages utilize distinct algorithms, as detailed in \cite{XinTodd2023} and \cite{de2009ehrhartTodd}, respectively. For example, the prob8 instance is processed by \texttt{LattE} to transform $62044$ unimodular cones
in $6.0$ seconds, whereas \texttt{DecDenu} transforms $139188$ unimodular cones in $1.2$ seconds. We observe that the computation time for both packages is nearly insensitive to the value of $i$. It is also important to mention that the method outlined in \cite{aardal2002hard} is solely capable of ascertaining whether $d(F+i;a_1,\dots, a_n) > 0$.

In the initial 12 hard instances, \texttt{DecDenu} generates a larger number of cones than \texttt{LattE}, yet its computation time is notably less. For the subsequent 10 random instances, \texttt{DecDenu} yields fewer cones than \texttt{LattE}, but its computation time is approximately one-twentieth that of \texttt{LattE}.

{\tiny
\begin{table}[htbp]
  \centering
    \begin{tabular}{c|c c| c c |c c}
    & \multicolumn{2}{c}{LattE} & \multicolumn{2}{c}{DecDenu} & \multicolumn{2}{c}{LattE/DecDenu} \\ \hline
        & \# of unimodular cone & time(s) & \# of unimodular cone & time(s) & Number ratio & Time ratio \\ \hline
    cuww4 & 364 & 0.01 & 1036 & 0.012 & 0.351 & 0.8333333 \\
    cuww5 & 2514 & 0.16 & 5548 & 0.041 & 0.453 & 3.902 \\
    prob1 & 10618 & 0.85 & 24786 & 0.178 & 0.428 & 4.775 \\
    prob2 & 6244 & 0.51 & 11072 & 0.077 & 0.564 & 6.623 \\
    prob3 & 12972 & 1.01 & 11490 & 0.08 & 1.129 & 12.625 \\
    prob4 & 9732 & 0.72 & 15438 & 0.112 & 0.630 & 6.429 \\
    prob5 & 8414 & 0.73 & 29595 & 0.221 & 0.284 & 3.303 \\
    prob6 & 26448 & 2.82 & 52916 & 0.409 & 0.500 & 6.895 \\
    prob7 & 20192 & 2.54 & 43552 & 0.35 & 0.464 & 7.257 \\
    prob8 & 62044 & 7.56 & 139188 & 1.09 & 0.446 & 6.936 \\
    prob9 & 36354 & 4.07 & 69808 & 0.547 & 0.521 & 7.441 \\
    prob10 & 38638 & 3.88 & 53766 & 0.42 & 0.719 & 9.238 \\ \hline
    prob11 & 5600284 & 701 & 4455683 & 36.7 & 1.257 & 19.101 \\
    prob12 & 10778801 & 1417 & 6961202 & 57.9 & 1.548 & 24.473 \\
    prob13 & 8400814 & 1076 & 6085420 & 52.4 & 1.380 & 20.534 \\
    prob14 & 9068188 & 1221 & 7026995 & 59.3 & 1.290 & 20.590 \\
    prob15 & 7647072 & 944 & 5183979 & 43.7 & 1.475 & 21.602 \\
    prob16 & 6242020 & 818 & 4921562 & 40.9 & 1.268 & 20.000 \\
    prob17 & 9000574 & 1237 & 6519150 & 54.2 & 1.381 & 22.823 \\
    prob18 & 10097931 & 1309 & 6450759 & 53.4 & 1.565 & 24.513 \\
    prob19 & 8506304 & 1074 & 6041508 & 49.8 & 1.408 & 21.566 \\
    prob20 & 9568498 & 1248 & 6527133 & 54.6 & 1.466 & 22.857 \\\hline
    \end{tabular}%
     \caption{Data for \texttt{LattE} and \texttt{DecDenu}. \label{table_result_cuww}}
\end{table}%
}

We conclude this section with the following remarks.

The multiplier chosen by \texttt{DecDenu} may not always be optimal. For instance, let $\alpha=(13429,26850, 26855 ,40280 ,40281 ,53711, 53714, 67141)$ be
the vector in problem cuww5. In the computation of $f^m(\ci 2;\alpha)$, the LLL algorithm suggests the multiplier $2447$, resulting in
$$  f^m(\ci 2;\alpha) \overset{\m=2447}{=} f^m(\ci 1; 26850,3637, 12235, 1190, 1257, 67, 7408, 1123)$$
and then \texttt{DecDenu} produces $2226$ leaves.
In contrast, using the multiplier $1$ yields
\begin{align*}
   f^m(\ci 2;\alpha) \overset{\m=1}{=} &f^m(\ci 1; 26850, 13421, 5, 13420, 13419, 11, 14, 13409) \\
   =&\sum_{i=2}^{8} f^m(\ci i; 26850, 13421, 5, 13420, 13419, 11, 14, 13409)
\end{align*}
and then \texttt{DecDenu} gives $910$ leaves.

We can seamlessly transition to \texttt{LattE} at any node of $\TD$, but we are unable to switch to \texttt{DecDenu} at nodes in $\TB$.
\begin{enumerate}
  \item When encountering a bad vertex in $\TD$, we have the option to switch to \texttt{LattE}.

  \item In all of the finite cases where $a_1 \leq 13$ and $d<10$, the proportion of cases where $nl(\TD)>nl(\TB)$ is less than $1/10$. It is feasible to log these instances and employ \texttt{LattE} when such cases arise.
\end{enumerate}

\section{LLLCTEuclid: A Novel Algorithm for Constant Term Extraction}
We commence by elucidating the domain of the iterated Laurent series and the foundational elements essential for this section. Subsequently, we delineate the \texttt{LLLCTEuclid} algorithm, an innovative application of our \texttt{DecDenu} algorithm for the extraction of constant terms from Elliott rational functions. The \texttt{LLLCTEuclid} algorithm is grounded in the principles of \texttt{CTEuclid} but leverages the LLL's algorithm extensively. This integration positions it as a formidable tool for constant term extraction, making it a compelling choice among elimination-based algorithms.

\subsection{The field of iterated Laurent series}
Iterated Laurent series provides a powerful framework for handling multivariate rational functions, particularly in the context of constant term extraction. For any field \( K \), the field \( K((x)) \) of Laurent series is constructed as \(\{\sum_{k=N}^{\infty} a_k x^k: N\in \mathbb{Z}, a_k\in K\}\). The field of iterated Laurent series can be recursively defined, and it was introduced by the first author in \cite{xin2004fast} as a foundational tool for constant term extraction.

In this sub-section, we provide a succinct overview of the iterated Laurent series. Our working field is \( K = \mathbb{Q}((x_m))((x_{m-1}))\cdots((x_1)) \), which encompasses the field of rational functions as a subfield. This field is specified by the sequence of variables \( (x_1,\dots, x_m) \), and the order of the variables is crucial in the series expansion of a rational function.

\emph{Unique Series Expansion in \( K \).} Working within \( K \), every element possesses a unique series expansion. Consequently, the constant term operators \( \CT_{x_i} \) are well-defined and commute with each other, as evidenced by the relations \( \CT_{x_i}\CT_{x_j}=\CT_{x_j}\CT_{x_i} \). These commute relations are indispensable. Importantly, there is no convergence issue; the constant term of an iterated Laurent series remains an iterated Laurent series.

We are particularly interested in Elliott-rational functions, which are rational functions with binomial factors only in their denominator. Elliott-rational functions are named after Elliott's reduction identity, which demonstrates that the constant term of an Elliott rational function is still an Elliott rational function.
The expansion of an Elliott rational function in \( K \) is straightforward yet pivotal. Every monomial \( M \neq 1 \) can be compared with \( 1 \) in \( K \) by identifying the ``smallest" variable \( x_j \) appearing in \( M \), i.e., \( \deg_{x_i} M = 0 \) for all \( i < j \). If \( \deg_{x_j} M > 0 \), then \( M \) is considered ``small", denoted \( M < 1 \); otherwise, \( M \) is ``large", denoted \( M > 1 \). For the unique series expansion of Elliott-rational functions in \( K \), it suffices to consider the following two types of series expansions:
\[
\frac 1 {1-M} = \left\{
  \begin{array}{ll}
    \sum_{k\geq0}M^k, & \text{ if } M < 1 ;\\
    \frac 1 {-M(1-1/M)} = \sum_{k\geq 0} -\frac 1 {M^{k+1}}, & \text{ if } M > 1.
  \end{array}
\right.
\]

When expanded as a series in \( K \), an Elliott-rational function \( E \) is typically expressed in its proper form:
\begin{equation}\label{f-Elliott}
E = \frac{\text{ a Laurent polynomial}}{
\prod_{j=1}^n (1-\text{Monomial}_j)}
=\frac{L}{
\prod_{j=1}^n (1-M_j)}=L\prod_{j=1}^n \Big(\sum_{k\ge 0} (M_j) ^k\Big),  \  (\textrm{proper form})
\end{equation}
where \( L \) is a Laurent polynomial and \( M_j < 1 \) for all \( j \). It is important to note that the proper form of \( E \) is not unique; for instance, \( 1/(1-x) = (1+x)/(1-x^2) \) are both proper forms.

\medskip
Now, we elucidate the \emph{core problem} within MacMahon's partition analysis. Let \( \Lambda = \{\lambda_1, \dots, \lambda_r\} = \{x_{i_1}, \dots, x_{i_r}\} \) be a subset of variables, and we employ the concise notation \( \Lambda^{\mathbf{b}} = \lambda_1^{b_1} \cdots \lambda_r^{b_r} \). We can then express
\begin{equation}\label{e-Lambda-Elliott}
E = \frac{L(\Lambda)}{\prod_{j=1}^n (1-\Lambda^{\mathbf{c}_j} M_j)},  \quad (\textrm{proper form in }K)
\end{equation}
where \( M_j \) are independent of \( \Lambda \) for all \( j \). The \emph{type} of \( E \) (with respect to \( \Lambda \)) is defined to be the matrix
\( A = (c_1, \dots, c_n) \). It is typically assumed that \( A \) is of rank \( r \), and \( d = n - r \) is referred to as the dimension of
\( \CT_{\Lambda} E \). If \( L(\Lambda) = \Lambda^{-\mathbf{b}} \), then
$$ \CT_{\Lambda} E = \sum_{\alpha \in P \cap \mathbb{N}^n} M_1^{\alpha_1} \cdots M_n^{\alpha_n},$$
where \( P \) is the polyhedron defined by \( P = \{\alpha \in \mathbb{R}^n : A\alpha = \mathbf{b}\}.\)

\noindent
{\bf Core Problem.}
Given an Elliott rational function \( E \) as described, the objective is to represent the constant term of \( E \) in \( K \) as follows:
\begin{align*}
\CT_{\lambda_1, \dots, \lambda_r} E & = \text{ a short sum of simple rational functions free of the } \lambda \text{-variables}.
\end{align*}

\subsection{The \texttt{LLLCTEuclid} Algorithm}
\begin{algor}[LLLCTEuclid]
\label{alg:LLLCTEuclid}
\mbox{ } \\
\textbf{Input:} An Elliott rational function \( E \) as in \eqref{e-Lambda-Elliott} and a working field \( K \).
\\
\textbf{Output:} \( \CT_{\Lambda} E \) as a short sum of simple rational functions.
\begin{enumerate}
  \item[1.1] Initialize
  \begin{equation}\label{eq_O0}
O_0' = \frac{L(\Lambda)} {(1-z_1M_1(x)\Lambda^{c_1})(1-z_2M_2(x)\Lambda^{c_2})\cdots (1-z_nM_n(x)\Lambda^{c_n})},
  \end{equation}
and work in \( K' = K((z_n))\cdots((z_1)) \), the field specified by \( (z_1, \dots, z_n, x_1, \dots, x_m) \).
  \item[1.2] (Optimization) By optimizing the matrix \( A \), find \( (A^L, U) \) such that i) \( U \) is a nonsingular \( r \times r \) matrix; ii) \( A^L = UA = (c_1', \dots, c_n') \) is an integral matrix; iii) the row vectors of \( A^L \) form a reduced basis. Let \( O_0 \) be obtained from \( O_0' \) by replacing \( \lambda_i \) with \( \Lambda^{Ue_i} \) for all \( i \), yielding
  $$O_0 = \frac {L(\Lambda^{Ue_1}, \dots, \Lambda^{Ue_r})} {(1-z_1M_1(x)\Lambda^{c_1'})(1-z_2M_2(x)\Lambda^{c_2'})\cdots (1-z_nM_n(x)\Lambda^{c_n'})},$$
  where the type of \( O_0 \) is now \( A^L \).
  \item[2.1] Compute \( O_{i+1} \) from \( O_i \) for \( i = 0, 1, \dots, r-1 \) as follows:
  \item[2.2] For each summand \( O_{ij} \) in \( O_i \), select a variable \( \lambda \) and apply \cite[Lemma 9]{xin2015euclid} to \( F(\lambda) = O_{ij} \) to express \( \CT_{\lambda}O_{ij} \) as a linear combination of \( \A_{1-u_i\lambda^{a_i}} F^{\pm}(\lambda) \) for some \( F^{\pm}(\lambda) \) of a similar form as \( F(\lambda) \).
  \item[2.3] (Optimization) Compute each \( \A_{1-u_i\lambda^{a_i}} F^{\pm}(\lambda) \) using Algorithm \ref{alg_DecDenu} (\texttt{DecDeun}) in Section \ref{sec_DecDenu}.
  \item[3] After eliminating all the variables \( \lambda_i \), we obtain a big sum
  $$O_r = \CT_{\Lambda}O_0 = \sum_i \frac{L_i(x;z)}{(1-M_{i1}z^{B_{i1}})(1-M_{i2}z^{B_{i2}})\cdots (1-M_{id}z^{B_{id}})}.$$
  Eliminate the slack variables \( z_j = 1 \) for all \( j \) using generalized Todd polynomials as in \cite{XinTodd2023}.
\end{enumerate}

\end{algor}

The novel algorithm, \texttt{LLLCTEuclid}, is constructed within the framework of \texttt{CTEuclid}, with modifications to Steps 1.2, 2.3, and 3. The computation of \( \A_{1-u_i\lambda^{a_i}} F^{\pm}(\lambda) \) in Step 2.3 is the focal point of this study, which we have addressed in Section \ref{sec_DecDenu}. Step 3 employs the log-exponential trick for generalized Todd polynomials, as detailed in our recent work \cite{XinTodd2023}.

We now elaborate on Step 1, which introduces a new simplification facilitated by the LLL algorithm. For simplicity, we consider the case where \( L(\Lambda) = \Lambda^{-\mathbf{b}} \) to illustrate the concept. The approach extends naturally for Laurent polynomial $L$. The inclusion of slack variables \( z_j \) in Step 1 ensures that the denominator factors of \( O_{ij} \) remain coprime throughout our elimination process. Consequently, Proposition \ref{prop-PFD-fund} can be applied consistently when computing \( \CT_{\lambda} O_{ij} \) in Step 2.

Using direct series expansion, it is evident that
$$\CT_{\Lambda} E = \Big(\CT_{\Lambda} O_0'\Big)\Big|_{z_i=1, i=1,2,\dots,n}.$$
The type of \( O_0' \) is \( A = (c_1, \dots, c_n) \), and the constant term corresponds to a weighted count of lattice points within the polytope \( P = \{\alpha \in \mathbb{R}^n : A\alpha = \mathbf{b} \}. \) Elementary row operations do not alter the solution set, thus \( P = \{\alpha \in \mathbb{R}^n : WA\alpha = W\mathbf{b} \} \) for any nonsingular matrix \( W \). We aim for \( WA \) to have small integral entries. Additionally, we seek to simplify equations such as \( k\alpha_1 - 2k\alpha_n = b \) to \( \alpha_1 - 2\alpha_n = \frac{b}{k} \), ensuring that \( P \) contains no lattice points when \( \frac{b}{k} \) is not an integer. This corresponds to dividing a row of \( WA \) by their greatest common divisor.

To utilize the aforementioned idea, we define the action of a nonsingular matrix \( W = (w_{ij})_{r \times r} \) by
$$WF(\Lambda) := W \circ F(\Lambda) = F(\gamma_1, \gamma_2, \dots, \gamma_n), \quad \text{where} \quad \gamma_i = \Lambda^{W e_i} = \lambda_1^{w_{1i}} \cdots \lambda_n^{w_{ni}}. $$
We allow \( W \) to possess rational entries, hence the right-hand side may include fractional powers. Notably, \( W\Lambda^{\alpha} = \Lambda^{W\alpha} \). This leads to the following variant of Jacobi's change of variable formula.
\begin{lem} \label{lem-tr}
Suppose \( W \) is a nonsingular matrix and \( F(\Lambda) \) is a Laurent polynomial in \( \Lambda \). Then we have
$$\CT_{\Lambda} F(\Lambda) = \CT_{\Lambda} (W \circ F(\Lambda)).
$$
\end{lem}
\begin{proof}
  By linearity, we can assume \( F(\Lambda) = \Lambda^\alpha \). Then \( W \circ F(\Lambda) = \Lambda^{W\alpha} \). Since \( \alpha = 0 \Leftrightarrow W\alpha = 0 \), the lemma follows.
\end{proof}
By linearity, this lemma extends to encompass all \( F \) that are series in other variables with coefficients that are Laurent polynomials in \( \Lambda \). For instance, $F$ is in the ring
\(\mathbb{Q}[x_1, \dots, x_m, x_1^{-1}, \dots, x_m^{-1}][[z_1, z_2, \dots, z_n]] \). In particular, the lemma applies to \( O_0' \). It is worth noting that the general Jacobi's change of variable formula involves constant terms in two fields. For instance,
\(\CT_{\lambda} \sum_{k=0}^{\infty} \lambda^k = \CT_{\lambda} \sum_{k=0}^{\infty} \lambda^{-k}\), but the right-hand side is not a Laurent series in \( \lambda \). For more details, see \cite{xin2005residue}.
\begin{prop}
  Let \( O_0' \) be as in \eqref{eq_O0} of type \( A_{r \times n} \) with rank \( s \). Then one can find an equivalent constant term \( \CT_\Lambda O_0' = \CT_\Lambda O_0 \) such that: i) \( O_0 \) is free of \( \lambda_{s+1}, \dots, \lambda_r \); ii) the type \( A^L \) of \( O_0 \) is of full rank \( s \); iii) \( A^L \) only has 1 as its \( s \)-th determinant factor; iv) the rows of \( A^L \) form a reduced basis.
\end{prop}
 \begin{proof}
We commence by employing the Smith normal form to identify unimodular matrices \( U \) and \( V \) such that \( U A V = S = \begin{pmatrix}
      \text{diag}(d_1, \dots, d_s) & 0 \\
      0 & 0
    \end{pmatrix} \). This yields \( A\alpha = 0 \Leftrightarrow S V^{-1} \alpha = 0 \Leftrightarrow k_i = 0 \) for \( i = 1, 2, \dots, s \), where \( V^{-1}\alpha = (k_1, \dots, k_n)^T \) is integral.
Subsequently, we define \( W_1 = \begin{pmatrix}
      \text{diag}(d_1^{-1}, \dots, d_s^{-1}) & 0 \\
      0 & E_{n-s}
    \end{pmatrix} U \) and \( A' = W_1 A = \begin{pmatrix}
      E_s & 0 \\
      0 & 0
    \end{pmatrix}
 V^{-1}. \)
Applying the lemma \ref{lem-tr}, we obtain
\begin{align*}
 \CT_{\Lambda} O_0' &=\CT_\Lambda \frac{W_1 L(\Lambda)} {W_1\circ (1-z_1M_1(x)\Lambda^{c_1})(1-z_2M_2(x)\Lambda^{c_2})\cdots (1-z_nM_n(x)\Lambda^{c_n})}
 \\
 &=\CT_{\lambda_1,\dots, \lambda_s} \frac{L'(\lambda_1,\dots, \lambda_s)}{W_1\circ (1-z_1M_1(x)\Lambda^{c_1})(1-z_2M_2(x)\Lambda^{c_2})\cdots (1-z_nM_n(x)\Lambda^{c_n})},
\end{align*}
where the denominator is devoid of \( \lambda_{s+1}, \dots, \lambda_r \), allowing us to compute \( L'(\lambda_1,\dots, \lambda_s) = \CT_{\lambda_{s+1},\dots, \lambda_r} W_1 L(\Lambda) \) straightforwardly.
We then define \( A'' \) as the first \( s \) rows of \( A' \), and its Smith normal form is \( UA''V =(E_s, 0) \), implying that the \( s \)-th determinant factor of \( A'' \) is \( 1 \). Consequently, for any unimodular \( W \), any row of \( WA'' \) has the greatest common divisor $1$.

Finally, we apply LLL's algorithm to \( A'' \) to identify a unimodular matrix \( W_2 \) such that \( A^L = W_2 A'' \) with row vectors forming a reduced basis. We then establish
$$ O_0 = W_2\circ \frac{L'(\lambda_1,\dots, \lambda_s)}{ W_1\circ (1-z_1M_1(x)\Lambda^{c_1})(1-z_2M_2(x)\Lambda^{c_2})\cdots (1-z_nM_n(x)\Lambda^{c_n})}, $$
which satisfies the desired conditions outlined in the proposition.
\end{proof}

\begin{exam}
Compute $\CT_\Lambda F$, where $F=F(\Lambda)$ has the following proper form:
$$F=\frac{
\frac{ \lambda_{1}^{11} \lambda_{4}^{33} }{\lambda_{2}^{33} \lambda_{3}^{39}}
 +\frac{\lambda_3^{32}}{\lambda_1^{4} \lambda_2^{112} \lambda_4^{7}}
+ \frac{\lambda_3^{111}}{\lambda_1^{15} \lambda_2^{358} \lambda_4^{45}}
}
{
\left(1-\frac{\lambda_{1}^{19} \lambda_{2}^{165} \lambda_{4}^{57} z_{1}}{\lambda_{3}^{99}}\right) \left(1-\frac{\lambda_{3}^{36} z_{2}}{\lambda_{1}^{8} \lambda_{2}^{30} \lambda_{4}^{24}}\right) \left(1-\frac{\lambda_{3}^{5} z_{3}}{\lambda_{1} \lambda_{2}^{7} \lambda_{4}^{3}}\right) \left(1-\frac{\lambda_{3}^{50} z_{4}}{\lambda_{1}^{10} \lambda_{2}^{72} \lambda_{4}^{30}}\right) \left(1-\frac{\lambda_{2}^{28} z_{5}}{\lambda_{3}^{4}}\right)
}.$$
\end{exam}
\begin{proof}[Solution]
Using \texttt{CTEuclid} produces 23 terms, while using \texttt{LLLCTEuclid} only produces 9 terms.
We only illustrate the first step of the \texttt{LLLCTEuclid}.

The type of $F$ is represented by the matrix
$$A=\left(\begin{array}{ccccc}
19 & -8 & -1 & -10 & 0
\\
 165 & -30 & -7 & -72 & 28
\\
 -99 & 36 & 5 & 50 & -4
\\
 57 & -24 & -3 & -30 & 0
\end{array}\right).$$
The Smith Normal Form of matrix $A$ is given by
$$ S=\left(\begin{array}{ccccc}
1 & 0 & 0 & 0 & 0
\\
 0 & 2 & 0 & 0 & 0
\\
 0 & 0 & 4 & 0 & 0
\\
 0 & 0 & 0 & 0 & 0
\end{array}\right)
 =UAV= \left(\begin{array}{cccc}
1 & 0 & 0 & 0
\\
 -4 & 1 & 1 & 0
\\
 -3 & 0 & 1 & 0
\\
 -3 & 0 & 0 & 1
\end{array}\right)
A\left(\begin{array}{ccccc}
-2 & 0 & -1 & -1 & -1
\\
 0 & 0 & 0 & 1 & 0
\\
 251 & 10 & 121 & 3 & 1
\\
 -29 & -1 & -14 & -3 & -2
\\
 0 & 0 & 0 & 0 & 1
\end{array}\right)
.$$
Let $W_1$ be defined as
$$W_1 = \textrm{diag}\left(1,\frac1 2,\frac 1 4,1\right) U =\left(\begin{array}{cccc}
1 & 0 & 0 & 0
\\
 -2 & \frac{1}{2} & \frac{1}{2} & 0
\\
 -\frac{3}{4} & 0 & \frac{1}{4} & 0
\\
 -3 & 0 & 0 & 1
\end{array}\right)
$$
and then the transformed matrix $A'$ is given by
$$A'= W_1 A =\left(\begin{array}{ccccc}
19 & -8 & -1 & -10 & 0
\\
 -5 & 19 & 1 & 9 & 12
\\
 -39 & 15 & 2 & 20 & -1
\\
 0 & 0 & 0 & 0 & 0
\end{array}\right).$$
Set $A''$  as the first \( 3 \) rows of \( A' \) and apply the LLL algorithm  on the rows of $A''$, we obtain
$$A^L =W_2 A'' =\left(\begin{array}{ccccc}
-1 & -1 & 0 & 0 & -1
\\
 0 & 0 & -1 & -1 & -1
\\
 2 & -1 & 0 & -1 & 0
\end{array}\right),$$
where $W_2=\left(\begin{array}{ccc}
2 & 0 & 1
\\
 -222 & -9 & -107
\\
 25 & 1 & 12
\end{array}\right).$
Hence, we derive the following transformation:
\begin{align*}
  \CT_\Lambda F = \CT_\Lambda W_2\circ (W_1\circ F) &=\CT_\Lambda \frac{
\lambda_{1}^{4} \lambda_{2}^{6} \lambda_{3}
+ \frac{\lambda_1^3 \lambda_4^5}{ \lambda_2}
+\frac{\lambda_1^9  }{\lambda_2^{3/2} \lambda_3^{1/2}}
}
{
 \left(1-\frac{\lambda_{3}^{2} z_{1}}{\lambda_{1}}\right) \left(1-\frac{z_{2}}{\lambda_{1} \lambda_{3}}\right) \left(1-\frac{z_{3}}{\lambda_{2}}\right) \left(1-\frac{z_{4}}{\lambda_{2} \lambda_{3}}\right) \left(1-\frac{z_{5}}{\lambda_{1} \lambda_{2}}\right)
 } \\
   & =\CT_\Lambda \frac{
\lambda_{1}^{4} \lambda_{2}^{6} \lambda_{3}
}
{
 \left(1-\frac{\lambda_{3}^{2} z_{1}}{\lambda_{1}}\right) \left(1-\frac{z_{2}}{\lambda_{1} \lambda_{3}}\right) \left(1-\frac{z_{3}}{\lambda_{2}}\right) \left(1-\frac{z_{4}}{\lambda_{2} \lambda_{3}}\right) \left(1-\frac{z_{5}}{\lambda_{1} \lambda_{2}}\right)
 }.
\end{align*}

\end{proof}

\noindent
{\small \textbf{Acknowledgements:}
This work is partially supported by the National Natural Science Foundation of China [12071311].}


\begin{thebibliography}{10}

\bibitem{aardal2002hard}
{\sc K.~Aardal and A.~K. Lenstra}, {\em Hard equality constrained integer
  knapsacks}, in Integer Programming and Combinatorial Optimization: 9th
  International IPCO Conference Cambridge, MA, USA, May 27--29, 2002
  Proceedings 9, Springer, 2002, pp.~350--366.

\bibitem{alfonsin2005diophantine}
{\sc J.~L.~R. Alfons{\'\i}n}, {\em The diophantine Frobenius problem}, OUP
  Oxford, 2005.

\bibitem{andrews2001macmahonOmegaalgorithm2}
{\sc G.~E. Andrews, P.~Paule, and A.~Riese}, {\em {MacMahon's partition
  analysis VI: a new reduction algorithm}}, Annals of Combinatorics, 5 (2001),
  pp.~251--270.

\bibitem{bajo2024weighted}
{\sc E.~Bajo, R.~Davis, J.~A. De~Loera, A.~Garber, S.~G. Mora, K.~Jochemko, and
  J.~Yu}, {\em Weighted ehrhart theory: Extending stanley's nonnegativity
  theorem}, Advances in Mathematics, 444 (2024), p.~109627.

\bibitem{baldoni2013coefficients}
{\sc V.~Baldoni, N.~Berline, J.~De~Loera, B.~Dutra, M.~Koeppe, and M.~Vergne},
  {\em {Coefficients of Sylvester's Denumerant}}, arXiv preprint
  arXiv:1312.7147,  (2013).

\bibitem{Baldoni2012}
{\sc V.~Baldoni, N.~Berline, J.~A. De~Loera, M.~K{\"o}ppe, and M.~Vergne}, {\em
  {Computation of the highest coefficients of weighted Ehrhart
  quasi-polynomials of rational polyhedra}}, Foundations of Computational
  Mathematics, 12 (2012), pp.~435--469.

\bibitem{barvinokwodd2003short}
{\sc A.~Barvinok and K.~Woods}, {\em Short rational generating functions for
  lattice point problems}, Journal of the American Mathematical Society, 16
  (2003), pp.~957--979.

\bibitem{barvinok1994polynomialalgorithm}
{\sc A.~I. Barvinok}, {\em A polynomial time algorithm for counting integral
  points in polyhedra when the dimension is fixed}, Mathematics of Operations
  Research, 19 (1994), pp.~769--779.

\bibitem{breuer2017polyhedral}
{\sc F.~Breuer and Z.~Zafeirakopoulos}, {\em {Polyhedral Omega: a new algorithm
  for solving linear diophantine systems}}, Annals of Combinatorics, 21 (2017),
  pp.~211--280.

\bibitem{de2005computational}
{\sc J.~A. De~Loera, D.~Haws, R.~Hemmecke, P.~Huggins, and R.~Yoshida}, {\em {A
  computational study of integer programming algorithms based on Barvinok's
  rational functions}}, Discrete Optimization, 2 (2005), pp.~135--144.

\bibitem{de2009ehrhartTodd}
{\sc J.~A. De~Loera, D.~C. Haws, and M.~K{\"o}ppe}, {\em Ehrhart polynomials of
  matroid polytopes and polymatroids}, Discrete \& computational geometry, 42
  (2009), pp.~670--702.

\bibitem{de2004effectiveLattE}
{\sc J.~A. De~Loera, R.~Hemmecke, J.~Tauzer, and R.~Yoshida}, {\em Effective
  lattice point counting in rational convex polytopes}, Journal of symbolic
  computation, 38 (2004), pp.~1273--1302.

\bibitem{DONGRE202339}
{\sc P.~Dongre, B.~Drabkin, J.~Lim, E.~Partida, E.~Roy, D.~Ruff, A.~Seceleanu,
  and T.~Tang}, {\em Computing rational powers of monomial ideals}, Journal of
  Symbolic Computation, 116 (2023), pp.~39--57.

\bibitem{cl2013LLL}
{\sc C.~Ling, W.~H. Mow, and N.~Howgrave-Graham}, {\em Reduced and
  fixed-complexity variants of the lll algorithm for communications}, IEEE
  Transactions on Communications, 61 (2013), pp.~1040--1050.

\bibitem{stanley2011EC1}
{\sc R.~P. Stanley}, {\em Enumerative Combinatorics Volume I second edition},
  Cambridge University Press, 2011.

\bibitem{xin2004fast}
{\sc G.~Xin}, {\em {A fast algorithm for MacMahon's partition analysis}}, The
  Electronic Journal of Combinatorics, 11 (2004), p.~1.

\bibitem{xin2005residue}
\leavevmode\vrule height 2pt depth -1.6pt width 23pt, {\em {A residue theorem
  for Malcev--Neumann series}}, Advances in Applied Mathematics, 35 (2005),
  pp.~271--293.

\bibitem{xin2015euclid}
\leavevmode\vrule height 2pt depth -1.6pt width 23pt, {\em {A Euclid style
  algorithm for MacMahon's partition analysis}}, Journal of Combinatorial
  Theory, Series A, 131 (2015), pp.~32--60.

\bibitem{xinxudedekindsums2023}
{\sc G.~Xin and X.~Xu}, {\em {A polynomial time algorithm for calculating
  Fourier-Dedekind sums}}, arXiv preprint arXiv:2303.01185,  (2023).

\bibitem{xin2023algebraic}
{\sc G.~Xin and C.~Zhang}, {\em {An algebraic combinatorial approach to
  Sylvester's denumerant}}, arXiv preprint arXiv:2312.01569,  (2023).

\bibitem{XinTodd2023}
{\sc G.~Xin, Y.~Zhang, and Z.~Zhang}, {\em Fast evaluation of generalized todd
  polynomials: Applications to macmahon's partition analysis and integer
  programming}, arXiv preprint arXiv:2304.13323,  (2023).

\end{thebibliography}
\end{document}